\documentclass[12pt,leqno]{amsart}

\usepackage{amsmath,amssymb,amscd,amsthm}
\usepackage{latexsym}
\usepackage[dvips]{graphicx,epsfig}
\usepackage{graphicx}
\usepackage{color}
\usepackage{fourier}

\newtheorem{theorem}{Theorem}
\newtheorem{lemma}{Lemma}
\newtheorem{proposition}{Proposition}

\newtheorem{definition}{Definition}
\newtheorem{corollary}{Corollary}

\newtheorem{claim}{Claim}

\newcommand{\f}[2]{\frac{#1}{#2}}
\newcommand{\dpr}[2]{\langle #1,#2 \rangle}

\newcommand{\cj}{{\mathcal J}} 
\newcommand{\cv}{{\mathcal V}} 
\newcommand{\al}{\alpha}

\newcommand{\ga}{\gamma}

\newcommand{\de}{\delta}

\newcommand{\la}{\lambda}

\newcommand{\si}{\sigma}

\newcommand{\rone}{\mathbf R}

\DeclareMathOperator{\sech}{sech}

\newcommand{\cl}{\mathcal L}

\newcommand{\ck}{\mathcal K}

\newcommand{\dint}{\displaystyle\int}

\newcommand{\df}{\displaystyle\frac}
\newcommand{\p}{\partial}

\newcommand{\beq}{\begin{equation}}
\newcommand{\eeq}{\end{equation}}
\newcommand{\beqna}{\begin{eqnarray*}}
\newcommand{\eeqna}{\end{eqnarray*}}
\newcommand{\beqn}{\begin{equation*}}
\newcommand{\eeqn}{\end{equation*}}
\newcommand{\bp}{\begin{proof}}
\newcommand{\ep}{\end{proof}}
\newcommand{\bprop}{\begin{proposition}}
\newcommand{\eprop}{\end{proposition}}
\newcommand{\bt}{\begin{theorem}}
\newcommand{\et}{\end{theorem}}
\newcommand{\bex}{\begin{Example}}
\newcommand{\eex}{\end{Example}}
\newcommand{\bc}{\begin{corollary}}
\newcommand{\ec}{\end{corollary}}
\newcommand{\bcl}{\begin{claim}}
\newcommand{\ecl}{\end{claim}}
\newcommand{\bl}{\begin{lemma}}
\newcommand{\el}{\end{lemma}}

\begin{document}

\title
[On the Barashenkov-Bogdan-Zhanlav solitons and their stability]
{On the Barashenkov-Bogdan-Zhanlav solitons and their stability}

\author{Wen Feng}
\author{Milena Stanislavova}
\author{Atanas G. Stefanov*}

\address{Wen Feng, Department of Mathematics, Niagara University, 5795 Lewiston Rd., Niagara University,  NY 14109 }
\email{wfeng@niagara.edu}

\address{Milena Stanislavova, Department of Mathematics, University of Alabama - Birmingham,  1402 10th Avenue South, 
	Birmingham AL 35294-1241  }
\email{mstanisl@uab.edu}

 \address{Atanas G. Stefanov,  Department of Mathematics, University of Alabama - Birmingham,  1402 10th Avenue South, 
 	Birmingham AL 35294-1241 }
 \email{stefanov@uab.edu}

\thanks{  Feng is partially supported by a graduate fellowship under grant \# 1516245. 
Stanislavova is partially supported by NSF-DMS, Applied Mathematics program, under
grant \# 1516245. Stefanov is partially supported by NSF-DMS, under grant \# 1908626.}

\date{\today}

\subjclass[2000]{35Q55, 35Q60,35B32}

\keywords{solitons, stability, Lugiato-Lefever model}
\begin{abstract}
The Barashenkov-Bogdan-Zhanlav solitons $u_\pm$ for the forced NLS/Lugiato-Lefever model on  the line are considered.  While the instability of $u_+$ was established in the original paper, \cite{B1}, the analogous question for $u_-$ was only considered heuristically and numerically. 
We rigorously analyze the stability of $u_-$ in the various regime of the parameters. In particular, we show that $u_-$ is spectrally stable for small pump strength $h$. Moreover, $u_-$ remains spectrally stable until a pair of neutral eigenvalues of negative Krein signature hits another pair of eigenvalues, which has emanated from the edge of the continuous spectrum, \cite{B1, BBK, ABP}.   After the collision, an  instability is conjectured and numerically observed in previous works, \cite{B1}. 
\end{abstract}

\maketitle

\section{Introduction}
Optical combs generated by micro-resonators is an active area of research, \cite{CY, KHD,  Matsko,MR}, see also \cite{Holz} for reports on concrete experimental data. As a physical process having to do with electromagnetism, the relevant starting point is the Maxwell equation. We refer the interested reader to consult  \cite{LL1} for physical derivation in the important case of a cavity filled with medium obeying the Kerr's law. The proposed mechanism of pattern formations and the related discussions on various parameters are also presented   in great detail.  There are numerous papers dealing with the model derivation, as well as  reductions to dimensionless variables, see for example \cite{CM},\cite{LL1}, \cite{MR}. 

    In this article, our starting point of investigation is the consensus model in one spatial dimension, namely  the Lugiato-Lefever equation.  In normalized  variables, it may be written as 
    \begin{equation}\label{lug}
   i u_t  + u_{xx} - u+ 2|u|^2u=-i \gamma u-h, x\in \rone, 
    \end{equation}
    where $u$ is the field envelop, $t$ is the normalized time, $x$ is the retarded and normalized  coordinate, $\gamma$ is the normalized damping/detuning  rate, and $h$ is the normalized pump strength.   Note that both $\ga, h$ are real parameters. 
    \subsection{Steady states} 
    The time-independent solutions of \eqref{lug}  satisfy the elliptic PDE
    \begin{equation}
    \label{10} 
    -u''+ u - 2|u|^2 u=i \ga u + h. 
    \end{equation}
    Further reducing the problem, we consider   the damping free case, that is $\gamma =0$.  Such  problem is  referred to as forced (or driven) NLS equation,  which has a different physical interpretation, \cite{B1, BS}.   The study of the existence and stability of the steady states \eqref{10} were analyzed for different ranges of the parameters $\ga, h$ in the following works, \cite{KN, B1, NB, Ter, BS}. 

The main subject of our investigations will be about  the time independent solutions, for the forced NLS problem.  
In this case, the steady state problem, that is \eqref{10},  takes the simpler form 
\begin{equation}
\label{11} 
-u''+ u - 2|u|^2 u= h. 
\end{equation}
Clearly, one cannot expect, for $h\neq 0$, the solutions to \eqref{11} to decay at $\pm \infty$, in fact the terminal value, denoted  $\lim_{x\to \pm \infty}u_{\pm}(x)=\psi_0$ must satisfy the cubic equation 
$2\psi_0^3-\psi_0+h=0$. As it happens, these problems have been studied before, see \cite{B1}.  In fact, there are  the following explicit solutions, which have appeared  repeatedly in the literature.   
\begin{proposition}
	\label{prop:10} 	
	Let $h\in (0, \f{2}{3\sqrt{6}})$ and $\al\in (0, \infty)$ is the unique real, so that 
	\begin{equation}
	\label{h:10} 
	h =      \frac{\sqrt{2} \cosh^2 \alpha}{(1+2 \cosh^2 \alpha)^{3/2}}.
	\end{equation}
	  Then, the  functions $u_\pm$, given by 
	\begin{eqnarray}
	\label{20} 
	u_{\pm}(x) &=& \psi_0(1+ \varphi_{\pm}(x))= \psi_0\left(  1+ \frac{2 \sinh^2 \alpha}{1 \pm \cosh \alpha \cosh( Ax)}\right), \\
	\label{60} 
		\psi_0 &=&  \frac{1}{\sqrt{2(1+2 \cosh^2 \alpha)}}, \ \ A=  \frac{\sqrt{2} \sinh \alpha }{\sqrt{1+ 2\cosh^2 \alpha}} 
	\end{eqnarray}
  are solutions to \eqref{11}.   
\end{proposition}
{\bf Remark:}  Note the important relations  $h=\psi_0-2\psi_0^3$ and $\psi_0^2<\f{1}{6}$, as a consequence of   \eqref{60}. As a result, the map $\psi_0\to h$ is a one-to-one correspondence between $(0, \f{1}{\sqrt{6}})$ and $(0,  \f{2}{3\sqrt{6}})$, whence the constraints in the statement of Proposition \ref{prop:10}. 

 It is a natural question to see whether such solutions are dynamically stable in the context of the forced NLS, that is  \eqref{lug}, with $\ga=0$. In fact, this question has been considered in \cite{B1}, where the authors have offered an analytical solution for the case $u_+$, while the case of $u_-$ was treated only numerically. 
 \subsection{Linearizations and stability} 
 \label{sec:2.1} 
 We first perform the linearization of  the system \eqref{lug} about the solutions $u_\pm$. 
 
 Consider perturbations in the form $u=u_{\pm} + y_1(x) +i y_2(x)$.  Plug this in \eqref{lug}, and after ignoring terms $O(z^2)$, we obtain the following linear problem  
 \begin{equation}
 \label{30} 
 \p_t \left(
 \begin{array}{c}
 y_1 \\
 y_2\\
 \end{array}
 \right)
 = 
 \left(
 \begin{array}{cc}
 0 &1\\
 -1 &0 \\
 \end{array} 
 \right)    
 \left(
 \begin{array}{cc}
 \cl_{+} &0\\
 0 &\cl_{-} \\
 \end{array} 
 \right)  
 \left(
 \begin{array}{c}
 y_1\\
 y_2\\
 \end{array} 
 \right) 
 \end{equation}
 where 
 \[ \begin{cases} 
  \cl_{+}=- \partial_{xx}-6u_\pm^2+1\\
 \cl_{-} =- \partial_{xx}-2u_\pm^2+1
 \end{cases}
 \]
 Introduce the eigenvalue ansatz $\left(
 \begin{array}{c}
 y_1 \\
 y_2\\
 \end{array}
 \right)\to e^{\la t} \left(
 \begin{array}{c}
 z_1 \\
 z_2\\
 \end{array}
 \right)$, the self-adjoint operator $\cl$ and the skew symmetric  $\cj$ 
 \begin{equation}
 \label{35} 
 \cl: =\left(
 \begin{array}{cc}
 \cl_{+} &0\\
 0 &\cl_{-} \\
 \end{array} 
 \right), \cj:=\left(
 \begin{array}{cc}
 0 &1\\
 -1 &0 \\
 \end{array} 
 \right)
 \end{equation}
 so that we can rewrite \eqref{30} in the compact form 
 \begin{equation}
 \label{40} 
 \cj \cl \vec{z}=\la \vec{z},
 \end{equation}
 which is the well-known Hamiltonian formulation of the eigenvalue problem.  
 \begin{definition}
 	\label{defi:10} 
 	We say that the wave $u_\pm$ is spectrally stable if the linearized operator $\cj \cl$ does not have spectrum in the 
 	right-hand complex plane. In other words $\si(\cj \cl) \subset \{\la: \Re \la\leq 0\}$.  Otherwise, the wave is referred to as spectrally unstable. 
 \end{definition}
 Following Weyl's theory\footnote{which asserts the stability of essential spectrum under suitable perturbations}, we use the standard split  of the spectrum into pure point spectrum and essential spectrum. More precisely, for a closed operator $A$ pure point spectrum $\si_{p.p.}(A)$ consists of eigenvalues of finite multiplicity of $A$, whereas the rest is essential spectrum, $\si_{ess.}(A)=\si(A)\setminus \si_{p.p.}(A)$. 
 
  Next, we present  our  main result. 
 \begin{theorem}
 	\label{theo:10} 
 	The solitons $u_+$ is spectrally unstable, with exactly one positive eigenvalue. 
 	
 	The steady state  $u_-$ on the other hand is spectrally stable for small values $h>0$, up to some critical value $h^*$. 
 	
   More precisely, there exists a value $\al^*\in [0, \infty)$, so that for all   $\al>\al^*$, the linearized operator  satisfies $\si(\cj\cl)\subset i\rone$. Note that  
 	$$
 	\si_{ess.}(\cj \cl)=\{i \la: \la\in \rone, |\la|\geq \sqrt{(1-6\psi_0^2(\al))(1-2\psi_0^2(\al))} \}.
 	$$
 	Furthermore, $\si_{p.p.}(\cj\cl)$ has a multiplicity  two eigenvalue at zero, due to translational symmetry\footnote{and as such has algebraic multiplicity two and geometric multiplicity one, see  Section \ref{sec:3} below for explicit descriptions of these} as well as a pair of simple, 
 	purely imaginary eigenvalues $\pm i \mu(\al)$ with  negative Krein signature. 
 	
 	Finally,   $\lim_{\al\to \infty} \mu(\al)=0$ and the eigenvalues $\pm i \mu(\al)$ collide, as $\al\to \al^*+$,  either with the edge of the continuous  spectrum 
 	$
 	\pm i \sqrt{(1-6\psi_0^2(\al))(1-2\psi_0^2(\al))} 
 	$
 	or with another pair of eigenvalues $\pm i \tilde{\mu}(\al)$, so that $0<\mu(\al)<\tilde{\mu}(\al)< \sqrt{(1-6\psi_0^2(\al))(1-2\psi_0^2(\al))}$, which are of positive Krein signature. That is,  $0<\mu(\al)<\tilde{\mu}(\al):\lim_{\al\to \al^*+} \tilde{\mu}(\al) - \mu(\al)=0$. 
  	In other words, we have a pair of neutral eigenvalues $\pm i \mu(\al)$, 
 		which travels from $0$ (corresponding to  $h=0$ and  $\al=\infty$) to $\mu(\al^*)$, 
 	and  $h\to  h^*=h(\al^*)$. 
 \end{theorem} 
{\bf Remarks:} 
\begin{itemize}
	\item The results in \cite{B1} already contain rigorous analysis for the instability of $u_+$. 
	\item Regarding the stability of $u_-$, some heuristic arguments were presented in \cite{B1}, which were complemented by numerical simulations. 
	\item We do not present a rigorously established  mechanism for the instability formation. However, from our arguments,  it is confirmed that the waves remain spectrally stable, till the pair of purely imaginary/neutral eigenvalues $\pm i \mu(\al)$ hits the edge of the continuous  spectrum or the pair $\pm i \tilde{\mu}(\al)$. According to the numerical simulations in \cite{B1} (but also the perturbative calculations in \cite{BBK}, which confirm that a pair of purely  imaginary eigenvalues is peeled off the essential spectrum, as $h>0$),  an instability is triggered by a  collision of the pairs $\pm i \mu(\al)$  and $\pm i \tilde{\mu}(\al)$, after a which a quartet of eigenvalues (of which two have a negative real part, while the other two have positive real part, hence create instability) is formed in the complex plane. 
	\item According to the numerics in \cite{B1}, the second alternative occurs, namely $\pm i \mu(\al)$ hits another eigenvalue $\pm i \tilde{\mu}(\al)$ and exits the imaginary axes after as a modulational instability. In fact, the collision between $i \mu(\al)$ and $i \tilde{\mu}(\al)$ happens at $\al^*\sim 2.5327$ or $h^*\sim 0.07749$, \cite{B1}. 
\end{itemize}
 We plan on presenting the proof of Theorem \ref{theo:10}, which consists of several different claims,  in 
	the following steps. We show that $u_+$ is spectrally unstable in Section \ref{sec:3.1}. The claims about $\si_{ess.}(\cj\cl)$ is in Lemma \ref{le:14}. Regarding the wave $u_-$, the claims about the point spectrum $\si-{pt.}(\cj\cl)$ for small $h$ are in Section \ref{sec:3.2}, the precise result is stated in Proposition \ref{prop:54}. Then, the tracking for larger values of $h$, more specifically the alternative for  developing eventual complex instabilities, is explored in Proposition \ref{prop:99}.  In preparation for this, Proposition \ref{prop:98} however shows that no instability occurs, before 1) collision with another eigenvalue or 2) collision with continuous spectrum.   
 \section{Preliminaries}
 We now discuss the basics of the instability index theory. 
\subsection{Instability index theory} 
We use the instability index count theory, as developed in \cite{KKS,KKS2}. We present a corollary, which is enough  for  our purposes. For eigenvalue problem in the form \eqref{40}, 
  we assume that $\cl$ has a finite number of negative eigenvalues, $n(\cl)$ and $\cj^{-1}: Ker[\cl]\to Ker[\cl]^\perp$.  

Let $k_r$  be  the number of positive eigenvalues of \eqref{40}, $k_c$ be the number of 4 tuples of eigenvalues with non-zero real and imaginary parts\footnote{as any eigenvalue $\la: \Re \la\neq 0, \Im \la\neq 0$ will join $\si_{p.p.}(\cj\cl)$, together with $-\la, \bar{\la}, -\bar{\la}$, due to Hamiltonian symmetries} and $k_i^-$, the number of pairs of purely imaginary eigenvalues with negative Krein signature.  For a simple pair of imaginary eigenvalues $\pm i \mu$, and the corresponding eigenvector 
$\vec{z} = \left(\begin{array}{c}
z_1  \\ z_2
\end{array}\right): \cj \cl \vec{z} = i \mu \vec{z}$, the Krein signature is 
$
sgn(\dpr{ \cl  \vec{z}}{ \vec{z}})
$
 see \cite{KKS}, p. 267. That is, we say that the signature is negative, if $\dpr{ \cl  \vec{z}}{ \vec{z}}<0$. We note that since  the self-adjoint operator $\cl: Ker[\cl]^\perp\to Ker[\cl]^\perp$, we can define properly $\cl^{-1}: Ker[\cl]^\perp\to Ker[\cl]^\perp$. 
 
We are now ready to introduce  a matrix $D$. Namely,  picking  a basis for  $Ker[\cl]$, say $Ker[\cl]=span \{\phi_1, \ldots, \phi_n\}$, set 
\begin{equation}
\label{dij}
D_{i j}:=\dpr{\cl^{-1} [\cj^{-1} \phi_i] }{\cj^{-1} \phi_j}.
\end{equation}
Note that the last formula makes sense, since $ \cj^{-1} \phi_i \in Ker[\cl]^\perp$, whence  $\cl^{-1}[\cj^{-1} \phi_i]\in Ker[\cl]^\perp$ is well-defined.  Note that we shall use the Morse index notation, namely for a self-adjoint, bounded from below operator $S$, with finitely many negative eigenvalues, denote 
$$
n(S)=\{\la<0: \la\in\si_{p.p.}(S)\},
$$
where the eigenvalues are counted with their respected multiplicities. 
The index counting theorem, see Theorem 1, \cite{KKS2} states that if $det(D)\neq 0$, then
\begin{equation}
\label{e:20}
k_r+2 k_c+2 k_i^-= n(\cl)-n(D). 
\end{equation}
  Note that the purely imaginary eigenvalues with negative Krein signatures play an important role in the instability formation - one manifestation of that is the formula \eqref{e:20}.  For example, one observes  that the law \eqref{e:20} allows for configurations  with $k_i^-=1, k_c=0$ that may  be transformed, as parameters vary,  into a case where $k_i^-=0, k_c=1$. 
 
 Indeed, a well-established mechanism of generation of instabilities is the collision of a  eigenvalue of negative Krein signature with a eigenvalue of positive Krein signature. Such collisions, may (and usually do) give birth to  a pair of complex eigenvalues, one with positive real part (hence the instability) and one with a negative one. We identify this below as a potential   mechanism of instability, and numerics in \cite{B1} indeed confirm that this is the case. 

Next, we discuss some specific spectral results about the linearized operators involved in the eigenvalue problem \eqref{40}.
\subsection{Some preliminary spectral results}
\begin{lemma}
	\label{le:14} 
	The essential spectrum of $\cj \cl$ is given by 
		$$
			\si_{ess.}(\cj\cl) = \{i\la:\la\in \rone,  |\la|\geq \sqrt{(1-6\psi_0^2(\al))(1-2\psi_0^2(\al))}\}.
		$$
\end{lemma}
\begin{proof}
Since $\lim_{x\to \pm \infty} u_\pm \to \psi_0$, we can write 	the operators (with $u=u_+$ or $u=u_-$)
\begin{eqnarray*}
\cl_+ &=&-\p_{xx}+1-6 u^2=-\p_{xx}+1-6\psi_0^2 - 6V\\
\cl_- &=& -\p_{xx}+1-2 u^2=-\p_{xx}+1-2\psi_0^2 - 2V
\end{eqnarray*} where $V=u^2-\psi_0^2$ have  exponential decay at $\pm \infty$. By Weyl's theorem, 
$$
	\si_{ess.}(\cj\cl) = \si(\cj \left(\begin{array}{cc} 
	-\p_{xx}+1-6\psi_0^2 & 0 \\
	0 & -\p_{xx}+1-2\psi_0^2
	\end{array}\right)).
$$
	In anticipation that the spectrum is inside of $i \rone$, 
	we set up the spectral  problem as the non-invertibility of the matrix operator 
	$
	\cj \left(\begin{array}{cc} 
	-\p_{xx}+1-6\psi_0^2 & 0 \\
	0 & -\p_{xx}+1-2\psi_0^2
	\end{array}\right) - i \la ,
	$
	which is equivalent to the non-invertibility of 
	$\left(\begin{array}{cc} 
	-\p_{xx}+1-6\psi_0^2 & 0 \\
	0 & -\p_{xx}+1-2\psi_0^2
	\end{array}\right)+i \la \cj$.
	By Fourier transform arguments, we need 
	$$
	\det\left(\begin{array}{cc} 
	k^2+1-6\psi_0^2 & i \la \\
	-i \la & k^2+1-2\psi_0^2
	\end{array}\right)=0,
	$$
	for some $k\in\rone$.  Note that due to the restriction $\psi_0^2<\f{1}{6}$(see \eqref{60}), the diagonal entries of the matrix are positive for each $k\in\rone$. 
	
Thus, $i\la\in  \si(\cj \left(\begin{array}{cc} 
	-\p_{xx}+1-6\psi_0^2 & 0 \\
	0 & -\p_{xx}+1-2\psi_0^2
	\end{array}\right))$ if and only if for some $k\in\rone$, 
	$$
	\la^2=(k^2+1-6\psi_0^2)(k^2+1-2\psi_0^2), 
	$$
	In other words, the continuous spectrum fills up the imaginary axes, with the exception of the segment from 
	$(-i \sqrt{(1-6\psi_0^2)(1-2\psi_0^2)}, i \sqrt{(1-6\psi_0^2)(1-2\psi_0^2)})$. 
	$$
	\si_{ess.}(\cj\cl)=\{\pm i \la: \la\geq \sqrt{(1-6\psi_0^2)(1-2\psi_0^2)}\}. 
	$$
\end{proof}
 {\bf Remark:} It may be an interesting exercise  to write down the spectrum of $\si(\cj\cl)$ for the case $\al=0$. In such a case, we are dealing with constant coefficient Schr\"odinger operators $\cl_\pm$, with $i\la\in \si(\cj\cl)=\si_{ess.}(\cj\cl)$ given by  $\la^2=k^2(k^2+\f{2}{3}), k\in \rone$, whence we obtain the formula
	$
	\si(\cj\cl)=\si_{ess.}(\cj\cl)=i \rone. 
	$

The next issue that we need to address is about the solvability of a  linear  problem of the type 
\begin{equation}
\label{e:10} 
(\cj\cl - i \mu) z=f,
\end{equation}
where the spectral parameter $i \mu$  is outside of the continuous  spectrum range, i.e. assuming that  $\mu \in (-\sqrt{(1-6\psi_0^2)(1-2\psi_0^2)},\sqrt{(1-6\psi_0^2)(1-2\psi_0^2)})$. 

We have the following Fredholm alternative type statement for the linear problem \eqref{e:10}. 
\begin{lemma}
	\label{le:22} 
	Let $\mu \in (-\sqrt{(1-6\psi_0^2)(1-2\psi_0^2)},\sqrt{(1-6\psi_0^2)(1-2\psi_0^2)})$. Let $\cj, \cl$ are as in \eqref{35} and $i\mu\in \si_{p.p.}(\cj\cl)$ is a simple eigenvalue, with an eigenfunction $z_0: (\cj \cl - i \mu) z_0=0$. 
	
	    Given $f\in L^2(\rone)$,  the  linear problem \eqref{e:10} has a solution, if and only if 
	    $\dpr{f}{\cj z_0}=0$. 
\end{lemma}
\begin{proof}
	The necessity of this condition is easy, since if we have solution of \eqref{e:10}, it suffices to take dot product of it with $\cj z_0$. We obtain 
	$$
	\dpr{f}{\cj z_0}=\dpr{(\cj\cl - i \mu) z}{\cj z_0}=\dpr{z}{\cl z_0+i \mu \cj z_0}=0. 
	$$
	The sufficiency part relies on the Fredholm properties of the operators. More specifically, write 
	$$
	\cl_0=\left(\begin{array}{cc} 
	-\p_{xx}+1-6\psi_0^2 & 0 \\
	0 & -\p_{xx}+1-2\psi_0^2
	\end{array} \right), \ \ \cv=\left(\begin{array}{cc} 
	6 V_+ & 0 \\
	0 & 2 V_-
	\end{array} \right), 
	$$
	so that \eqref{e:10} can be recast in the equivalent form 
	\begin{equation}
	\label{e:12} 
	(\cl_0 + i \mu\cj -\cv)z=-\cj f, 
	\end{equation}
	Due to the fact that $|\mu|<\sqrt{(1-6\psi_0^2)(1-2\psi_0^2)}$, we have that 
	$\cl_0 + i \mu\cj$ is invertible, so we can further rewrite \eqref{e:12} equivalently as 
	\begin{equation}
	\label{e:15} 
	(I -(\cl_0 + i \mu\cj)^{-1}\cv)z=-(\cl_0 + i \mu\cj)^{-1} \cj f=:\tilde{f}
	\end{equation}
	Now, due to the fact that $\cv$ is smooth and exponentially decaying (matrix) potential, while $(\cl_0 + i \mu\cj)^{-1}:L^2(\rone)\times L^2(\rone) \to H^2(\rone)\times H^2(\rone)$, with exponentially decaying kernel,  we have that $\ck:=(\cl_0 + i \mu\cj)^{-1}\cv:L^2\times L^2 \to L^2\times L^2$ is a compact operator. As such, the operator equation \eqref{e:15} is in the Fredholm alternative form $(id-\ck) z= \tilde{f}$. Therefore, it has solution, if $\tilde{f}\perp Ker(I-\ck^*)$.
	
	 We claim that under our assumptions, $Ker(I-\ck^*)=span[(\cl_0+i \mu \cj) z_0]$. Indeed, let $z^*\in Ker(I-\ck^*)$. We have that 
	$
	z^*=\cv (\cl_0 + i \mu\cj)^{-1} z^*
	$
	Letting $\eta^*:=(\cl_0 + i \mu\cj)^{-1} z^*$, we conclude that 
	$
	(\cl_0+i \mu \cj)\eta^*=\cv \eta^*,
	$
	whence 
	$$
	(\cj\cl-i \mu)\eta^*=(\cj(\cl_0-\cv) - i \mu)\eta^*=0.
	$$
	Since we have assumed that $i\mu$ is a simple eigenvalue, it follows that $\eta^*=c z_0$, whence $Ker(I-\ck^*)=span[(\cl_0+i \mu \cj) z_0]$. Thus, the solvability condition can be written as 
	$$
	0=\dpr{\tilde{f}}{(\cl_0+i \mu \cj) z_0}=\dpr{(\cl_0+i \mu \cj) \tilde{f}}{z_0}=-\dpr{\cj f}{z_0}=\dpr{f}{\cj z_0}. 
	$$
	
\end{proof}

Next, we provide  some properties of the linearized operators $\cl_\pm$, which will be useful in the sequel. 
\subsection{Properties of the linearized operators $\cl_\pm$} 
Note that $\cl_\pm$ are standard self - adjoint \\ Schr\"odinger operators with even potentials vanishing at $\infty$. 
In fact, since 
$u_\pm\to \psi_0$ as $x\to \pm \infty$, we have that $\cl_+ =-\p_{xx}+1- 6\psi_0^2 -6V_+$, 
$\cl_-=-\p_{xx}+1-2\psi_0^2 -2V_-$, whence 
$$
\si_{ess.} (\cl_+)= [1-6\psi_0^2, +\infty), \si_{ess.} (\cl_-)= [1-2\psi_0^2, +\infty) 
$$
Note that due to $\psi_0^2<\f{1}{6}$ in Proposition \ref{prop:10}, it follows that  $\si_{ess.}(\cl_\pm)\subset[1-6\psi_0^2, \infty)\subset (0, \infty)$. 

We have the following Proposition,  which collects some  pertinent spectral properties of $\cl_\pm$. 
\begin{proposition}
	\label{prop:20} 
	Let $h\in (0, \f{2}{3\sqrt{6}})$. Then, the  operators $\cl_\pm$ are self-adjoint, with domain $H^2(\rone)$. In addition, 
	$\si_{ess.}(\cl_\pm) \subset (0, \infty)$, $\cl_+[u'_\pm]=0$, so that $0\in \si_{p.p.}(\cl_+)$. 
	\begin{itemize}
		\item For the case $u_+$, $\cl_+$ has exactly one negative eigenvalue, while $\cl_->0$. 
		\item 
		For the case $u_-$, $\cl_+$ has exactly one negative eigenvalue,  while $\cl_-$ also has exactly one negative eigenvalue, and $0\notin\si_{p.p.}(\cl_-)$. 
	\end{itemize}
\end{proposition}
\begin{proof}
	Differentiating the profile equation \eqref{11}, implies 
	that $\cl_+[u_+']=0$. In addition, $u_+'$ has exactly one zero, at $x=0$. By Sturm-Liouville's criteria, zero is a simple eigenvalue, which is the second smallest eigenvalue.  So, there is exactly one negative eigenvalue for $\cl_+$. 
	
	Consider the case $u=u_+$. 
For  the operator $\cl_-$, clearly $\cl_->\cl_+$, so $\cl_-$ has at most one negative eigenvalue.  
  We will show that\footnote{this has already been proved in \cite{B1}, but we provide a direct, independent proof herein.} 
$n(\cl_-)=0$. 
	Note that the profile equation \eqref{11} is equivalent to $\cl_-[u_+]=h>0$. Assume for a contradiction that for some $\eta: \cl_-\eta=-\si^2 \eta$, $\si>0$. Since $\eta$ will be ground state for $\cl_-$, it follows that $\eta>0$ and $\eta$  will have exponential decay, in fact $\eta(x)\leq C e^{-|x| \sqrt{1-2\psi_0^2} }$. 
	Informally, we obtain the contradiction as follows 
	$$
	0 <h\int_{\rone} \eta(x) dx= \left<h, \eta \right> =\left< \cl_{-}u_{+}, \eta\right> = \left< u_{+}, \cl_{-} \eta \right> = -\sigma^2\left<u_{+}, \eta \right> < 0,
	$$
	since $u_+>0$. 
	
	Formally, fix a cut-off function, say $\zeta\in C^\infty_0(\rone), \zeta>0:  supp \zeta\subset (-2,2), \zeta(x)=1, |x|<1$ and a large real $N$. Let $\zeta_N(x):=\zeta(x/N)$. Compute 
	\begin{eqnarray*}
		0< h\dpr{1}{\eta}= \lim_{N\to \infty} \dpr{\zeta_N\cl_- u_+}{\eta} = \lim_{N\to \infty} 
		\dpr{u_+}{\cl_-[\zeta_N \eta]}.
	\end{eqnarray*} 
	Now, since 
	$$
	\cl_-[\zeta_N \eta]=\zeta_N \cl_-[\eta]-2 \zeta_N' \eta' -  \zeta_N''  \eta = 
	-\si^2 \zeta_N \eta+O_{L^2} (N^{-1}), 
	$$
	as $\zeta_N', \zeta_N''=O(N^{-1})$. 
	We compute 
	$$
	\lim_{N\to \infty}  	\dpr{u_+}{\cl_-[\zeta_N \eta]} = -\sigma^2\left<u_{+}, \eta \right> <0,
	$$
	which is a contradiction. The case, $\si=0$ is also contradictory, by the same argument, now that we know that there are no negative eigenvalues and  zero  must be the bottom of the spectrum, again an impossibility. 
	
	For the case $u=u_-$, we have again $n(\cl_-)\leq 1$, since $n(\cl_+)=1$ and $\cl_+<\cl_-$. 
	On the other hand,   by direct inspection, 
	$$
	\cl_{-} u_{-}  = 2 \psi_0^2 u_{-}(2+ u_{-})
	$$
	whence we can convince ourselves that 
	\begin{equation}
	\label{w:10} 
		\dpr{\cl_-u_-}{u_-}  = 2 \psi_0^2 \int_{-\infty}^{\infty}   u_{-}^2(2+ u_{-})~dx <0,
	\end{equation}
	Indeed, by computations aided by Mathematica, we were  able to  explicitly calculate 
 \begin{align*}
&  \dint_{-\infty}^{\infty} u_{-}^2(2+ u_{-})~dx =-\sqrt{2}\sinh^3(\alpha)\cosh(\alpha)\sqrt{\cosh(2\alpha)+2}\\
 &\times	\left( \df{2\coth(\alpha)csch(\alpha)}{\cosh^2(\alpha)}+2\pi \coth(\alpha)csch^2(\alpha) +\df{2\coth(\alpha)csch(\alpha)}{\cosh^2(\alpha)}+  2\pi \coth(\alpha)csch^2(\alpha) \right)\\
 &= -\sqrt{2}\sinh^3(\alpha)\cosh(\alpha)\sqrt{\cosh(2\alpha)+2} 
 \left(  \df{4\coth(\alpha)csch(\alpha)}{\cosh^2(\alpha)} +4\pi\coth(\alpha)csch^2(\alpha)\right)<0.
 \end{align*}
	
	Thus, $\cl_-$ has a negative eigenvalue and so $n(\cl_-)=1$. Finally, we claim that $\cl_-$ does not have eigenvalue at zero.    Before we start with our contradiction argument, let us point out that since the Schr\"odinger operator has $\cl_+=-\p_{xx}+1 - 6  u_-^2$ and  $u_-^2$ is an even, positive and decaying on $(0, \infty)$ function (i.e. bell-shaped), we conclude that its ground state, is bell-shaped as well.  
	
	Assume now for a contradiction that $0$ is an eigenvalue  $\cl_-[Q]=0$. This will be the second smallest eigenvalue for $\cl_-$, whence it will have exactly one zero, so it will be an odd function, vanishing at zero. So, in particular, $Q: \|Q\|=1$ will be perpendicular to the (bell-shaped) ground state for $\cl_+$.   But now recall $\cl_+<\cl_-$, so we have 
	$
	\dpr{\cl_+ Q}{Q}<\dpr{\cl_- Q}{Q}=0.
	$
	Thus, by Rayleigh formulas 
	$$
	\la_0(\cl_+)<\la_1(\cl_+)\leq  \dpr{\cl_+ Q}{Q}<0, 
	$$
	so $n(\cl_+)\geq 2$, in contradiction with what we know, namely $n(\cl_+)=1$.   Thus, $\cl_-$ does not have a zero eigenvalue, so it is in particular invertible operator. 
\end{proof}
We now describe the spectral properties of the linearized operators at $h=0$.   This is a well-known result, due to  M. Weinstein, \cite{W}, but the reader might consult  the excellent presentation in Section 4.1.1, \cite{SS}. It is convenient to utilize the notion of a generalized kernel of an operator, defined as the subspace $gKer(A):=span[Ker(A), Ker(A^2), \ldots]$. 
\begin{proposition}
	\label{prop:18} 
	For $h=0$, the cubic NLS has the following behavior of the linearized operators: 
	\begin{itemize}
		\item The operator $\cl_+$ has a single and simple negative eigenvalue, a simple eigenvalue at zero, with eigenfunction $u_0'$. $\cl_+$ is strictly positive on the co-dimension two subspace orthogonal to these two directions. 
		\item The operator $\cl_-$ has a simple eigenvalue at zero, spanned by $u_0$. It is positive on the co-dimension one subspace orthogonal to it. 
		\item The operator $\cj\cl$ has 
		$$
		\si_{ess.}(\cj\cl)=\{\pm i\la: \la\in\rone,  |\la|\geq 1\}, \si_{p.p.}(\cj\cl)=\{0\},
		$$
		where zero is an eigenvalue of algebraic  multiplicity four and geometric multiplicity two, generated by the translational and the modulational invariance. More precisely, 
		\begin{eqnarray*}
		 Ker[\cj\cl] &=& Ker[\cl]=span[\left(\begin{array}{c}
			u'  \\ 0
		\end{array}\right), \left(\begin{array}{c}
		0  \\ u
	\end{array}\right)]; \\
  gKer[\cj\cl] &=& span[\left(\begin{array}{c}
	0 \\ \cl_-^{-1}[u']
\end{array}\right), 
\left(\begin{array}{c}
	\cl_+^{-1}[u]  \\ 0
\end{array}\right)]. 
		\end{eqnarray*}
	\end{itemize}
\end{proposition}
In the arguments in the sequel, we use $h$ as a bifurcation parameter.   As the eigenvalues depend in a $C^1$ way on the parameter $h$, they  may move as $h$ changes. Specifically, for $h=0$, the eigenvalue zt zero is of algebraic multiplicity four (as to respect the translational and modulational invariance). The moment the parameter $h$ is turned on, the modulational invariance is broken, and a pair of eigenvalues separates from zero, in a smooth way. 
\section{Proof of Theorem \ref{theo:10}} 
\label{sec:3} 
We start our considerations with the proof for the instability of $u_+$. This has been previously  established in \cite{B1}, we provide the short argument here for completeness. 
\subsection{The instability of $u_+$} 
\label{sec:3.1}

The instability of $u_+$ is now an easy consequence of the results of Proposition \ref{prop:20} and the index count formula \eqref{e:20}. Indeed, on the right hand side of \eqref{e:20}, we have $n(\cl)=n(\cl_+)+n(\cl_-)=1+0=1$. Thus, the stability is determined by $n(D)$. Since in this case $Ker[\cl_-]=\{0\}$, we have that $D$ is a matrix of one element, namely  $\dpr{\cl_-^{-1} u_+'}{u_+'}$. However, since $\cl_->0$, we see that $\dpr{\cl_-^{-1} u_+'}{u_+'}>0$, whence $n(D)=0$, whence a single real instability is detected by \eqref{e:20}.

\subsection{The case of $u_-$: tracking the modulational eigenvalues as $0<h<<1$} 
\label{sec:3.2}
For $h=0$, we trivially settle on the  standard Schr\"odinger model, where all spectral information, including the spectrum of   $\si(\cj \cl)$, is well-known, see Proposition \ref{prop:18}.  More specifically, the operator $\cj \cl$ at $h=0$ has the structure of the spectrum as described in Proposition \ref{prop:18}, namely two eigenvectors and two generalized eigenvectors co-exist there.

 After turning on the parameter $h$, i.e. the moment $h\neq 0$, 
the ``translational eigenvalue'' pair 
$
\left(\begin{array}{c}
u'  \\ 0
\end{array}\right)$ and its corresponding generalized eigenvector $ \left(\begin{array}{c}
0 \\ \cl_-^{-1}[u']
\end{array}\right)$ persists, due to the fact that translational invariance is still intact, even after adding the $h$ in the model. Modulational invariance is however broken, once $h\neq 0$, so the other pair starts moving away from zero.   Note that the stability of the waves, or equivalently the eigenvalue problem  \eqref{40},  is completely determined by the behavior of this pair of eigenvalues, which at $h=0$ correspond to the modulational invariance. Indeed, as we observed in Lemma \ref{le:14}, $\si_{ess}(\cj\cl)\subset i \rone$. Thus, the wave is spectrally stable, that is, $\si(\cj\cl)\subset i \rone$ if and only if the modulational eigenvalue (of multiplicity two)  at $h=0$ split as a pair of purely  imaginary eigenvalues. We focus on the proof of this fact, for the case of the waves $u_-$. 

To this end,  we look at the right-hand side of \eqref{e:20}. It is clear that 
$n(\cl)=n(\cl_+)+n(\cl_-)=1+1=2$, while $n(D)=0$, so $n(\cl)-n(D)=2$. Thus, according to the indices on the left-hand side, we are presented with the following alternatives:  we either have two different positive unstable eigenvalues for $\cj \cl$ or we have a four tuple of eigenvalues (two of which are unstable), so $k_c=1$  or we have a pair of purely imaginary eigenvalues, with negative Krein signature (and so $k_i^-=1$, $k_r=k_c=0$).  We now present some heuristical argument on why it must be that the two negative Krein signature eigenvalues happen. Let us refute the other two cases - First, the case of two different positive unstable eigenvalues is not viable - this is still a Hamiltonian problem and this will effectively generate four eigenvalues (the two positive and the corresponding two with opposite signs), while we have only two eigenvalues  unaccounted for - namely the two previously modulational eigenvalues, which in the case $h>0$ may start moving, due to the broken modulational invariance. In fact, and for the same reason, even the case of a four tuple of eigenvalues cannot happen, because this creates four eigenvalues, in addition who are still sitting at the zero, for a total of six eigenvalues, whereas we started with four eigenvalues at $h=0$. 

We will show that for small $h$, a pair of purely imaginary eigenvalues, with negative Krein signature appear. 
According to \eqref{e:20}, such a configuration is spectrally stable.   Looking at the alternatives, it suffices to show that a pair of purely imaginary eigenvalues appears close to zero, and then they must necessarily be with negative Krein signatures. 

Before we continue with the construction of the modulational eigenvalues as $h\neq 0$, let us compute $n(D)$. Recall that according to Proposition \ref{prop:20}, we have that $Ker[\cl_-]=\{0\}$, while $Ker[\cl_+]=span[u_-']$. We claim that there is no another generalized eigenvector behind  $
\left(\begin{array}{c}
0 \\ \cl_-^{-1}[u']
\end{array}\right)$. Indeed, otherwise, we would have the solvability of the relation 
$$
\cj \cl \left(\begin{array}{c}
z_1 \\ z_2
\end{array}\right)=\left(\begin{array}{c}
0 \\ \cl_-^{-1}[u_-']
\end{array}\right), 
$$
Solving directly, this means that $z_2=0$, while $\cl_+[z_1]=-\cl_-^{-1}[u_-']$. This then would require a consistency relation 
$\dpr{\cl_-^{-1}[u_-']}{u_-'}=0$, which   is false. In fact, we show that $\dpr{\cl_-^{-1}[u_-']}{u_-'}>0$, for all values of $h$, see below. 

All in all, it turns out that $D$ has only one element, namely $\dpr{\cl_-^{-1}[u'_-]}{u'_-}$. Now, it is not as straightforward as in the classical case to conclude that  $\dpr{\cl_-^{-1}[u_-']}{u_-'}>0$, since $\cl_-$ is not a non-negative operator anymore, since in fact $n(\cl_-)=1$.  On the other hand, its ground state, say 
	$W:\cl_-[W]=-\si^2 W, \|W\|=1$  is bell-shaped (since  $\cl_-=-\p_{xx}+1 - 2 u_-^2$ is a Schr\"odinger operator with bell-shaped potential). Hence by parity considerations  $\dpr{u_-'}{W}=0$, so $u_-'\perp W$. But then, note that $\si(\cl_-^{-1}|_{\{W\}^\perp})\subset (0, \infty)$, whence $\cl_-^{-1}|_{\{W\}^\perp}>0$, whence 
$$
\dpr{\cl_-^{-1}[u'_-]}{u'_-}= \dpr{\cl_-^{-1}|_{\{W\}^\perp} [u'_-]}{u'_-}>0. 
$$
Thus $n(D)=0$ and in addition recall that  this was also useful in establishing that the Jordan block of $
\left(\begin{array}{c}
u'  \\ 0
\end{array}\right)$ is of length two. 

We now turn to the construction of the modulational eigenvalues for $h\neq 0$. Similar to the spectral problem in \cite{HSS}, we set up an ansatz as follows. 
\begin{equation}
\label{100}
\cj \left(\begin{array}{cc}
\cl_+^0+h V_+  & 0 \\
0 & \cl_-^0+h V_- 
\end{array}\right)\left(\begin{array}{c}
\sqrt{h} \psi_1  \\ u_0+h \psi_2
\end{array}\right)=i \mu_0 \sqrt{h} 
\left(\begin{array}{c}
\sqrt{h} \psi_1  \\ u_0+h \psi_2
\end{array}\right)+O\left(\begin{array}{c}
 h^{3/2}  \\ h^2
\end{array}\right)
\end{equation}
where, we have used the fact that $\cl_\pm^h=\cl_\pm^0+h V_\pm$, where $\cl_\pm^0$ are the standard Schr\"odinger operators $\cl^0_+=-\p_x^2+1 - 6 \sech^2(x), \cl^0_-=-\p_x^2+1 - 2 \sech^2(x)$, $u_0=\sech(x)$. Note $\cl_-^0[u_0]=0$. Also, the potentials $V_\pm$ can be explicitly written down, but this will not be necessary for our arguments. 
 Resolving \eqref{100} yields, to leading order in $h$, 
 \begin{eqnarray*}
 & & \cl_-^0 \Psi_2+V_- u_0 = i \mu_0 \Psi_1 \\
 & & -\cl_+^0 \Psi_1 = i \mu_0 u_0.
 \end{eqnarray*}
 Clearly, from the second equation, we need $\Psi_1= - i \mu_0 \cl_+^{-1}[u_0]$, which we then plug in the first equation. This is justified, since 
 $u_0\perp Ker[\cl_+]=span[u_0']$. It remains  to solve 
 $$
  \cl_-^0 \Psi_2=\mu_0^2 \cl_+^{-1}[u_0]-V_- u_0
 $$
This of course gives a solvability condition, namely $\dpr{u_0}{\mu_0^2 \cl_+^{-1}[u_0]-V_- u_0}=0$, which  is actually an equation for $\mu_0$. We obtain 
\begin{equation}
\label{105} 
\mu_0^2=\f{\dpr{V_- u_0}{u_0}}{\dpr{\cl_+^{-1} u_0}{u_0}}. 
\end{equation}
It is well-known (and also directly computable) that $\dpr{\cl_+^{-1} u_0}{u_0}<0$, as this is equivalent to the stability of the soliton $u_0$, as a solution to the Schr\"odinger equation. In fact, this is the Vakhitov-Kolokolov condition for stability of solitary waves, which is well-known to hold for the wave $u_0=sech(x)$. Unfortunately, $V_-$ is a sign-changing solution, so it is not immediately clear how to determine the sign of the quantity $\dpr{V_- u_0}{u_0}$. 
 Instead, we shall show by a roundabout argument that 
$\dpr{V_- u_0}{u_0}<0$. As a consequence,  \eqref{105} has a pair of real solutions 
\begin{equation}
\label{w:9} 
\mu_0=\pm \sqrt{\f{\dpr{V_- u_0}{u_0}}{\dpr{\cl_+^{-1} u_0}{u_0}}},
\end{equation}
 representing a pair of complex imaginary eigenvalues 
$\pm i \sqrt{\f{\dpr{V_- u_0}{u_0}}{\dpr{\cl_+^{-1} u_0}{u_0}}}$. 

Indeed, otherwise, if $\dpr{V_- u_0}{u_0}>0$, then we have constructed (in the form dictated by \eqref{100}) a pair of real eigenvalues for $\cj \cl$, namely $\pm \sqrt{-\f{\dpr{V_- u_0}{u_0}}{\dpr{\cl_+^{-1} u_0}{u_0}}}$, one stable, the other one unstable. But then, $n(\cl)-n(D)=2$, as established earlier, while on the left hand side of \eqref{e:20} $k_r=1$. This is impossible, a contradiction. Thus, $\dpr{V_- u_0}{u_0}<0$ and we have a pair of purely imaginary eigenvalues, with negative Krein signatures.  This shows the following proposition. 
\begin{proposition}
	\label{prop:54} 
	There exits $\al_0>>1$, so that for all $\al\in (\al_0, \infty)$, 
	the corresponding solutions $u_{-,\al}$ described in \eqref{20} are spectrally stable. 
	
	Moreover, the multiplicity four eigenvalue at zero for the standard NLS problem has transformed itself 
into an eigenvalue at zero with multiplicity two, 
	 and a pair $\pm i \mu(\al)$ is a pair of simple eigenvalues of negative Krein signatures, with even eigenfunctions. 
 
	In addition, $\mu:(\al_0,  \infty)\to \rone_+$ is decreasing and smooth  function, with $\lim_{\al\to \infty} \mu(\al)=0$ and $\mu(\al_0)>0$. 
\end{proposition}
\begin{proof}
	Basically, this is a perturbation argument about the standard NLS case, which corresponds to $h=0$, or equivalently $\al=+\infty$. In the narrative preceding the formal statement of Proposition \ref{prop:54}, we have shown that zero is still an eigenvalue, of multiplicity two, and we have also  constructed the eigenvalues $\pm i\mu(\al)$ for large values of $\al$. 
The only unproven claim in Proposition \ref{prop:54} is that $\mu$ is a decreasing function of $\al$ (hence increasing function of $h$), for large enough values of $\al$ (equivalently small enough values of $h$).  

 In order to see this monotonicity, and even though the dependence on variable $h$ is not smooth at $h=0$, we can express the formula \eqref{w:9} equivalently as 
 \begin{equation}
 \label{h:101} 
  \lim_{h\to 0+} \sqrt{h} \f{d\mu(h)}{h}=\f{1}{2} \sqrt{\f{\dpr{V_- u_0}{u_0}}{\dpr{\cl_+^{-1} u_0}{u_0}}}>0. 
 \end{equation}
 
 This shows that in a small neighborhood of $h=0$, $h\to \mu(h)$ is increasing. 
 \end{proof}
 \subsection{The soliton $u_-$: Tracking the neutral eigenvalues till the  collision}
 \label{sec:3.3}
 
 In Section \ref{sec:3.2}, we have demonstrated that for small values $0<h<<1$, the modulational eigenvalue (of multiplicity two) at $h=0$ splits into a pair of purely imaginary  eigenvalues $\pm i \mu(h), \mu(h)>0$ of negative Krein signature. We now wish to further track this pair as $h$ grows. Recall that by the smooth dependence on the parameters, the wave $u_-$ is spectrally  stable as long as these pair does not turn into a complex instability. We have described the mechanism of how this happens in the statement of our main result, Theorem \ref{theo:10}. We now provide the details of the proof. 
  
  Namely, we will show the following
  \begin{itemize}
  	\item as long as  $\mu(\al)<\sqrt{(1-6\psi_0^2(\al))(1-2\psi_0^2(\al))}$   and 
  	$\pm i \mu(\al)$ are simple eigenvalues (that is, we are in a pre-collision scenario), there is a neighborhood $(\al-\de, \al+\de)$, so that whenever $\tilde{\al}\in (\al-\de, \al+\de)$, the waves $u_{-, \tilde{\al}}$ are spectrally stable, with a pair of negative Krein signature eigenvalues $\pm i \mu(\tilde{\al})$. 
  	\item  There exists $\al^*\geq 0$, so that either $\lim_{\al\to \al^*+} \mu(\al)=\sqrt{(1-6\psi_0^2(\al^*))(1-2\psi_0^2(\al^*))}$ or there exists another family of eigenvalues $\pm i \tilde{\mu}(\al)$ for $\cj \cl_\al$ , with 
  	$$
  	\mu(\al)<\tilde{\mu}(\al): \lim_{\al\to \al^*} \tilde{\mu}(\al)-\mu(\al)=0. 
  	$$
  	Moreover, there exists $\si_0>0$, so that $\min_{\al^*<\al<\infty} \mu(\al)\geq \si_0>0$. 
  	That is, the pair $\pm i \mu(\al)$ potentially exits the imaginary axis either by hitting the edge of the essential spectrum or the pair  $\pm i \tilde{\mu}(\al)$ and stays a fixed distance away from zero. 
  \end{itemize}
   {\bf Remark:} According to the numerics in \cite{B1}, the second alternative occurs, namely $\pm i \mu(\al)$ hits another eigenvalue $\pm i \tilde{\mu}(\al)$ and exits the imaginary axes after as a modulational instability.
	
   Here and below, we use the parameters $h$ and $\al$ interchangeably, due to the one-to-one correspondence described explicitly in \eqref{h:10}. 
   \begin{proposition}
   	\label{prop:98} 
   	Let $\pm i \mu(\al)$ are the eigenvalues described in Proposition \ref{prop:54}, which are in the pre-collision mode. That is, $0<\mu(\al)<\sqrt{(1-6\psi_0^2(\al))(1-2\psi_0^2(\al))}$ and $\pm i \mu(\al)$ are simple. 
   	 Then, there exists $\de=\de(\al)>0$, so that  whenever $\tilde{\al}\in (\al-\de, \al+\de)$, the waves $u_{-, \tilde{\al}}$ are spectrally stable.   Moreover, the mapping $\al\to \mu(\al)$ is  $C^1(\al-\de, \al+\de)$.  
   \end{proposition}
   {\bf Remark:} Interestingly, the proof breaks down, if either $\mu(\al)=0$ or \\  $\mu(\al)=\sqrt{(1-6\psi_0^2(\al))(1-2\psi_0^2(\al))}$ or $ i \mu(\al)$ is not a simple eigenvalue. 
   \begin{proof}
   Fix  $\al$ is so that $0<\mu(\al)<\sqrt{(1-6\psi_0^2(\al))(1-2\psi_0^2(\al))}$ and  $ i \mu(\al)$ is  a simple eigenvalue. This is exactly the setup of Lemma \ref{le:22}, where $i\mu\notin \si_{ess.}(\cj \cl)$.  That is 
   	\begin{equation}
   	\label{e:47} 
   	\cj \cl z(\al)=i \mu(\al) z(\al),
   	\end{equation}
   	We now construct, under the assumptions imposed on $\al$  the eigenvalue in a neighborhood, $(\al-|\de|, \al+|\de|)$ for some small $\de: |\de|<<1$. First, we introduce the approximate operators 
   	\begin{eqnarray*}
   		\cl_+(\al+\de) &=& -\p_{xx}+1 - 6 u^2_{-,\al+\de}=-\p_{xx}+1 - 6 u^2_{-, \al}-\de V_+ +O(\de^2), \\
   		&=& \cl_+^0-\de V_++O(\de^2), \ \ V_+=12 u_{-,\al} \f{\p u_{-,\al}}{\p \al} \\
   		\cl_-(\al+\de) &=& -\p_{xx}+1 - 2 u^2_{-,\al+\de}=-\p_{xx}+1 - 2 u^2_{-, \al}-\de V_- +O(\de^2), \\
   		&=& \cl_-^0-\de V_-+O(\de^2), \ \ V_-=4 u_{-,\al} \f{\p u_{-,\al}}{\p \al}.
   	\end{eqnarray*}
   	Note that the potentials $V_\pm$ are sign-changing functions over $x\in (0, \infty)$. Introduce also an expansion in the  eigenvectors and the eigenvalues 
   	\begin{eqnarray*}
   		z(\al+\de) &=&z(\al)+\de q+O(\de^2)=:z_0+\de q+O(\de^2) \\
   		\mu(\al+\de) &=& \mu(\al)+\de r+O(\de^2)=:\mu_0+\de r+O(\de^2).
   	\end{eqnarray*}
     Using standard inverse function theorems, with 
   	$$
   	(z,\mu) \in \{(H^2(\rone)\times H^2(\rone))\times \rone: |z-z_0|<<1, |\mu-\mu_0|<<1\},
   	$$
   
   	 it  is enough  to show that $r,q$ can be selected so that the following system is solvable up to first order in $\de$: 
   	\begin{equation}
   	\label{e:60} 
   	\cj  \left(\begin{array}{cc} 
   	\cl^0_+ - \de V_+ & 0 \\
   	0 & \cl^0_--\de V_- 
   	\end{array}\right)(z_0+\de q) - i (\mu_0+\de r)(z_0+\de q)=0. 
   	\end{equation}
   	The order zero equations express the fact that \eqref{e:47} is satisfied.  The next order $\de$ equations yield  the following system 
   	\begin{equation}
   	\label{e:62} 
   	(\cj  \cl^0 - i \mu_0  )q=\cj \left(\begin{array}{cc} 
   	V_+ & 0 \\
   	0 &   V_- 
   	\end{array}\right) z_0 + i r z_0,
   	\end{equation}
   	which we need to check is solvable. 
   	Applying Lemma \ref{le:22}, matters reduce to verifying the solvability condition
   	\begin{equation}
   	\label{e:80} 
   	\dpr{\cj \left(\begin{array}{cc} 
   		V_+ & 0 \\
   		0 &   V_- 
   		\end{array}\right) z_0 + i r z_0}{\cj z_0}=0.
   	\end{equation}
   	This works out to an equation for $r$, which is 
   	\begin{equation}
   	\label{e:85} 
   	i r \dpr{z_0}{\cj z_0} = - \dpr{\left(\begin{array}{cc} 
   		V_+ & 0 \\
   		0 &   V_- 
   		\end{array}\right) z_0}{z_0}. 
   	\end{equation}
   	This has a solution, provided $ \dpr{z_0}{\cj z_0} \neq 0$. Once this is established, 
   	 we will be done with the proof of Proposition \ref{prop:98}. 
   	 To this end, from \eqref{e:47}, we have that $\cl_0z_0=-i\mu \cj z_0$, whence 
   	 $$
   	 i  \dpr{z_0}{\cj z_0}=\f{\dpr{\cl z_0}{z_0}}{\mu}<0,
   	 $$
   	 since $i \mu$ has negative Krein signature. 
   \end{proof}
   {\bf Remark:} The approximate formula \eqref{e:85} for $r=\mu'(\al)$ should,  in principle imply the expected  sign $\mu'(\al)<0$ (since we expect the mapping $\al\to \mu(\al)$ to be monotone decreasing, as $\mu(\infty)=0$ and it increases to $\mu(\al^*)>0$). Unfortunately, we cannot make a determination of the sign of the quantity $\dpr{\left(\begin{array}{cc} 
   	V_+ & 0 \\
   	0 &   V_- 
   	\end{array}\right) z_0}{z_0}$ based on our argument. 
  
  Our next goal is to establish that for some $\al^*\geq 0$, $\mu(\al)\in (0, \sqrt{(1-6\psi_0^2(\al))(1-2\psi_0^2(\al))})$ and $i \mu(\al)$ is a simple eigenvalue for all $\al\in (\al^*, \infty)$, while at least one of these two changes\footnote{According to the numerics in \cite{B1}, a collision with another eigenvalue occurs prior to   hitting the essential spectrum}  at $\al=\al^*$. Eventually, either $\al^*=0$ or $\al^*>0$ and  either   $i \mu(\al)$ collides with another eigenvalue $i \tilde{\mu}(\al)$ or 
  $
 \lim_{\al\to \al^{*}+} \mu(\al)=\sqrt{(1-6\psi_0^2(\al^*))(1-2\psi_0^2(\al^*))}).
  $

  \begin{proposition}
  	\label{prop:99} 
  	There exists $\al^*\geq 0$, so that $\mu(\al)\in (0, \sqrt{(1-6\psi_0^2(\al))(1-2\psi_0^2(\al))})$ and $i \mu(\al)$ is a simple eigenvalue,  for all $\al\in (\al^*, \infty)$. 
  	
  	Also, there exists $\si_0>0$, so that $\min_{\al^*<\al<\infty} \mu(\al)\geq \si_0>0$ and  either 
  	there exists a family of eigenvalues $\pm i\tilde{\mu}(\al): \mu(\al)<\tilde{\mu}(\al), \lim_{\al\to \al^*+} \tilde{\mu}(\al) - \mu(\al)=0$ or  
  	$$
  \lim_{\al\to \al^{*+}} \mu(\al)=\sqrt{(1-6\psi_0^2(\al^*))(1-2\psi_0^2(\al^*))}).
  	$$
  	In other words, $\pm i \mu(\al)$ (eventually) exits the imaginary axes (and becomes unstable)  after hitting another eigenvalue or through the edge of the continuous  spectrum. 
  \end{proposition}
  \begin{proof}
  	According to the results in Proposition \ref{prop:54}, we do not have to worry about the behavior of $\mu(\al)$ for very large $\al$, so it suffices to consider an interval $(0, \al_0)$, with $\al_0$ as in Proposition \ref{prop:54}. 
  	Due to the results of Proposition \ref{prop:98}, $\al\to \mu(\al)$ is a continuous function and we may define 
  	$$
  	\al^*=\inf\{	 \al >0:  i \mu(\al) \ \ \textup{is pre-collision}\}. 
  	$$
  	 Now, it is either the case that $\al^*=0$ or else $\al^*>0$. In the former case, there is nothing to do, while in the latter case, it remains to rule out the possibility that  $\lim_{\al\to \al^*+} \mu(\al)=0$. Let us note that when $\al^*>0$, there exists $\si_0>0$, so that 
  	 \begin{equation}
  	 \label{l:37} 
  	 \limsup_{\al\to \al_+} \la_0(\cl_{\pm,\al})<-\si_0, \ \liminf_{ \al\to \al_{+}}\la_1(\cl_{-, \al})>\si_0, 
  	   \liminf_{\al\to \al_+} \la_2(\cl_{+,\al})>\si_0.
  	 \end{equation}
  	 Indeed, \eqref{l:37} follows once we realize that for each compact interval 
  	 $J=[\al_1, \al_2]$, there is a constant $C=C_J$, so that for each $f\in H^2$ and $\nu_1, \nu_2\in J$, 
  	  \begin{equation}
  	  \label{l:38} 
  	 |\dpr{(\cl_{\pm,\nu_1}-\cl_{\pm,\nu_2})f}{f}|\leq C |\nu_1-\nu_2| \|f\|_{L^2}^2,
  	 \end{equation}
  	 as the operators $(\cl_{\pm,\nu_1}-\cl_{\pm,\nu_2})f=const(u_{\pm, \nu_1}^2-u_{\pm, \nu_2}^2)f=const(\nu_1-\nu_2) \p_\nu u_{\pm, \tilde{\nu}}^2f$. Then, once we have  \eqref{l:38}, we easily conclude that $\lim_{\al\to \al^*} \la_j(\cl_{\pm, \al})=\la_j(\cl_{\pm, \al^*}), j=0,1,\ldots$ and so on. Then, $ \limsup_{\al\to \al_+} \la_0(\cl_{\pm,\al})=
  	 \la_0(\cl_{\pm,\al^*})<0$, according to Proposition \ref{prop:20}. The other implications in   \eqref{l:37} follow in a similar manner. 
  	 
  	 Now, let us go back to the task at hand, namely to refute the possibility $\lim_{\al\to \al^*} \mu(\al)=0$. To this end, assume for a contradiction that in fact $\lim_{\al\to \al^*} \mu(\al)=0$. Consider the eigenvalue problem \eqref{e:47}. Note that by the construction in Proposition \ref{prop:98},  the eigenvalue problem is solved in the even subspace. In particular, the eigenvalue $\la_1(\cl_{+,\al})=0$ is not very relevant in our discussion as its eigenspace is spanned by an odd function $u_+'$. 
  	 
  	  	Introduce the real and imaginary parts $z^R:=\Re z(\al); z^I=\Im z(\al)$. Writing out the relation in \eqref{e:47} in terms of $z_1^R, z_1^I, z_2^R, z_2^I$  yields 
  	 \begin{eqnarray}
  	 \label{i:10}
  	 & &  \cl_{+, \al} z_1^R(\al)=\mu(\al) z_2^I(\al), \ \cl_{+, \al} z_1^I=-\mu(\al) z_2^R(\al) \\
  	 \label{i:20} 
  	 & & \cl_{-,\al} z_2^R(\al)=-\mu(\al) z_1^I(\al), \  \cl_{-,\al} z_2^I(\al)=\mu(\al)  z_1^R(\al).
  	 \end{eqnarray}
  	The eigenvectors $z(\al)\in D(\cl_\pm)=H^2$, so we normalize them as follows  $\|z(\al)\|_{L^2}=1$.  Denote the ground states by $\Psi_{\pm,\al}: \|\Psi_{\pm,\al}\|_{L^2}=1$, $\cl_{\pm,\al} \Psi_{\pm,\al}=\la_0(\cl_{\pm,\al}) \Psi_{\pm,\al}$. 
  In order to simplify the notations, we drop the dependence on $\al$. 
   	By taking $L^2$ norm in \eqref{i:10} and \eqref{i:20}, and applying \eqref{l:37}, we arrive at the estimates 
  	\begin{eqnarray}
  	\label{i:45} 
   |\la_0( \cl_{+})||\dpr{z_1^R}{\Psi_{+}}|\leq \|\cl_{+} z_1^R\|_{L^2}=\mu(\al) \|z_2^I\|_{L^2}.
  	\end{eqnarray}
  	Similarly, we establish 
  		\begin{eqnarray}
  		\label{i:47} 
  		& &   |\la_0(\cl_{+})| 	|\dpr{z_1^I}{\Psi_{+}}|\leq  \mu(\al) \|z_2^R\|_{L^2}, 	
  		|\la_0(\cl_{-})|	|\dpr{z_2^R}{\Psi_{-}}|\leq  \mu(\al) \|z_1^I\|_{L^2}, \\
  			\label{i:48} 
  		& & |\la_0(\cl_{-})|	|\dpr{z_2^I}{\Psi_{-}}|\leq  \mu(\al) \|z_1^R\|_{L^2}. 
  		\end{eqnarray}
  	By taking dot products in \eqref{i:10} and \eqref{i:20} with appropriate vectors, we have (with  $P_+ f:=f-\dpr{f}{\Psi_+} \Psi_+$, projecting over the positive subspace of $\cl_+$), 
  	\begin{eqnarray*}
  & & 	\dpr{\cl_{+} z_1^R}{z_1^R} = \la_0(\cl_+)\dpr{z_1^R}{\Psi_+}^2+
  	\dpr{\cl_+ P_+ z_1^R}{P_+ z_1^R}\geq \la_0(\cl_+)\dpr{z_1^R}{\Psi_+}^2+\si_0 \|P_+ z_1^R\|^2 \\
  	&\geq & \si_0 \|z_1^R\|_{L^2}^2 -(|\la_0(\cl_+)|+\si_0) \dpr{z_1^R}{\Psi_+}^2\geq \si_0 \|z_1^R\|_{L^2}^2 - \f{2}{\si_0} \mu(\al) \|z_2^I\|_{L^2}^2\geq \si_0 \|z_1^R\|_{L^2}^2 - \f{2}{\si_0} \mu(\al). 
  	\end{eqnarray*}
  	where we have used the estimate \eqref{i:45} and the normalization $\|z_2^I\|_{L^2}\leq \|z\|_{L^2}=1$.  Similarly, we establish 
  		\begin{eqnarray*}
  		& & 	\dpr{\cl_{+} z_1^I}{z_1^I} \geq  \si_0 \|z_1^I\|_{L^2}^2 - \f{2}{\si_0} \mu(\al), \\
  		& &  \dpr{\cl_{-} z_2^R}{z_2^R} \geq  \si_0 \|z_2^R\|_{L^2}^2 - \f{2}{\si_0} \mu(\al),\\
  		& & \dpr{\cl_{-} z_2^I}{z_2^I} \geq  \si_0 \|z_2^I\|_{L^2}^2 - \f{2}{\si_0} \mu(\al)
  		\end{eqnarray*}
  	Adding up all these estimates, together with the negative Krein signature,  implies 
  	\begin{eqnarray*}
  	0>\dpr{\cl z}{z}=\dpr{\cl_{+} z_1^R}{z_1^R}+	\dpr{\cl_{+} z_1^I}{z_1^I}+\dpr{\cl_{-} z_2^R}{z_2^R}+ \dpr{\cl_{-} z_2^I}{z_2^I}\geq \si_0 \|z\|^2- \f{8\mu(\al)}{\si_0}=
  	 \si_0 - \f{8\mu(\al)}{\si_0}
  		\end{eqnarray*}
  	This is clearly in contradiction with $\lim_{\al\to \al^*}\mu(\al) =0$. With this, Proposition \ref{prop:99} is established. 
  \end{proof}

\end{document}